\documentclass[12pt]{amsart}

\usepackage[colorlinks,linkcolor=blue,citecolor=blue,urlcolor=blue, pdfcenterwindow, pdfstartview={XYZ null null 1.2}, pdffitwindow, pdfdisplaydoctitle=true]{hyperref}
\usepackage{array}
\usepackage{threeparttable}
\usepackage{rotating}
\usepackage[section]{placeins}
\usepackage{amssymb} 
\usepackage{mathrsfs}
\usepackage{enumerate}
\usepackage{color}
\usepackage{amsrefs}        
\usepackage[all]{xy}

\input xy
\xyoption{all}
\begin{document}
\newtheorem{lemma}{Lemma}[section]
\newtheorem{lemm}[lemma]{Lemma}
\newtheorem{prop}[lemma]{Proposition}
\newtheorem{coro}[lemma]{Corollary}
\newtheorem{theo}[lemma]{Theorem}
\newtheorem{conj}[lemma]{Conjecture}
\newtheorem{prob}{Problem}
\newtheorem{ques}{Question}
\newtheorem{rema}[lemma]{Remark}
\newtheorem{rems}[lemma]{Remarks}
\newtheorem{defi}[lemma]{Definition}
\newtheorem{exam}[lemma]{Example}

\newcommand{\N}{\mathbb N}
\newcommand{\Z}{\mathbb Z}
\newcommand{\R}{\mathbb R}
\newcommand{\Q}{\mathbb Q}
\newcommand{\C}{\mathbb C}

\title{Bounded characteristic classes and flat bundles}

\author{I. Chatterji}
\author{Y. de Cornulier}
\author{G. Mislin}
\author{Ch. Pittet}

\address{MAPMO, Universit\'e d'Orl\'eans}    
\email{Indira.Chatterji@univ-orleans.fr}
\address{D\'epartement de Math\'ematiques, 
Universit\'e de Paris-Sud}
\email{yves.cornulier@math.u-psud.fr}
\address{Department of Mathematics, ETH, Z\"urich}    
\curraddr{Department of Mathematics, Ohio State University}
\email{mislin@math.ethz.ch}
\address{LATP, CMI, Aix-Marseille Universit\'e}
\email{pittet@cmi.univ-mrs.fr}

\keywords{Bounded cohomology,  classifying spaces, flat bundles, Lie groups, 
subgroup distortion, stable commutator length, word metrics}

\subjclass[2000]{Primary: 57T10, 55R40; Secondary: 20F65, 53C23, 22E15}


\thanks{The authors are grateful to the ETH Z\"urich, the MSRI in Berkeley, the Ohio State University, the Poncelet Laboratory in Moscow, where part of this work was done.
I. Chatterji and Y. de Cornulier were partially supported by the ANR JC08-318197.
Ch. Pittet was partially supported by the CNRS}

\date{February 15, 2012}

\begin{abstract}  
Let $G$ be a connected Lie group and let $R$ denote its radical. 
We show that all classes in the image of the
canonical map $H^*(BG,\mathbb R)\to H^*(BG^{\delta},\mathbb R)$ are bounded
if and only if the derived group $[R,R]$ is simply connected. We also give
equivalent conditions in terms of stable commutator length and distortion. 
\end{abstract}
\maketitle

\section{Introduction}

Let $G$ be a connected Lie group and let $G^{\delta}$ be the underlying group with the discrete topology. Let $A$ be
an abelian group endowed with a metric. The identity map $G^{\delta}\to G$ induces a natural ring homomorphism
\[
	H^*(BG,A)\to H^*(BG^{\delta},A),
\]
from the (singular) cohomology of the classifying space $BG$ of $G$ to the cohomology of $BG^{\delta}$. 
A class $\alpha$ in $H^n(BG^{\delta},A)=H^n(G,A)$ is \emph{bounded} if it can be represented by a cocycle $c:G^n\to A$ with bounded image.
We define a sub-additive function on $H^n(G,A)$ by
\[
	\|\alpha\|_{\infty}=\inf_{[c]=\alpha}\{\sup\{|c(g_1,\dots,g_n)|; g_1,\cdots,g_n\in G\}\}\;\in\mathbb R\cup\{\infty\}.
\]
When $A=\mathbb R$ with its usual metric, it corresponds to Gromov's seminorm \cite[Section 1.1]{GroVol} and was first considered by Dupont \cite{Dup} in degree two.
If $A$ is 
finitely generated, we use the word length with respect to a finite symmetric
generating set to define a metric. On $\mathbb R^r$ we consider the Euclidean metric
and when $A\cong\mathbb Z^r\subset\mathbb R^r$ is discrete we also consider the restriction of the Euclidean metric to $A$.  

The image of the natural map $H^*(BG,\mathbb R)\to H^*(BG^{\delta},\mathbb R)$ is
generated as a ring by bounded classes together with the image 
in degree two \cite[Lemma 51, Theorem 54]{CMPS}.

If a class $\alpha$ in the image of $H^2(BG,\mathbb R)\to 
H^2(BG^\delta, \mathbb R)$ is bounded, then
for any continuous map $\phi:\Sigma_g\to BG^\delta$ of a closed oriented surface of genus $g\ge 1$,
the characteristic number $\phi^*(\alpha)([\Sigma_g])\in \mathbb R$ belongs to a bounded
set $[-C,C]$ with $C=C(\alpha,g)$ depending only on $\alpha$ and $g$. Indeed,
$$|\phi^*(\alpha)([\Sigma_g])|\le ||\alpha||_\infty\cdot||[\Sigma_g]||_1=||\alpha||_\infty\cdot(4g-4),\quad
g\ge 1.$$ 
(For the
classical bound $C=g-1$ on the Euler number of flat ${\rm GL}^+_2(\mathbb R)$-bundles
over a surface of genus $g>0$, see Milnor \cite{Mil}.)
If, on the contrary, there exists a sequence $(\phi_n :\Sigma_g\to BG^\delta)$ of flat $G$-bundles with 
\[
	     \lim_{n\to\infty}\phi_n^*(\alpha)([\Sigma_g])=\infty,
\]
then $\alpha$ is unbounded.
(A family of maps $\phi_n: S^1\times S^1\to BQ^\delta$, $n\ge 1$, where $Q$ is the quotient of the three-dimensional Heisenberg
group by an infinite cyclic central subgroup, such that the sequence
$\phi^*_n(\alpha)([S^1\times S^1])$ is unbounded, where  $\alpha\in H^2(BQ^\delta,\mathbb R)$ is the image of a generator of $H^2(BQ,\mathbb Z)\cong\mathbb Z$,  is given in Goldman \cite{Gol}.) 

In the present work we give a necessary and sufficient condition on $G$ 
for the classes in the image of 
$H^*(BG,\mathbb R)\to H^*(BG^{\delta},\mathbb R)$ to be bounded.
Previously known sufficient conditions where: $G$ is linear algebraic (Gromov \cite[Theorem p. 23]{GroVol},
resp.~Bucher \cite{Bucher}), and the weaker condition that the radical $R$ of $G$ is linear
\cite[Theorem 54]{CMPS}; the radical of $G$ is the largest connected, normal,
solvable subgroup of $G$. It is well-known that $R$ is linear if and only if
the \emph{closure} of its derived subgroup $[R,R]$ is simply connected.
Hence linearity of $R$ is stronger than simple connectedness of $[R,R]$ because
$$\pi_1\left([R,R]\right)\subset\pi_1\left(\overline{[R,R]}\right).$$
For an example with 
$$\{0\}=\pi_1\left([R,R]\right)\subsetneq\pi_1\left(\overline{[R,R]}\right),$$
see \cite[Example 37]{CMPS}. 

Our main theorem shows that simple connectedness of $[R,R]$ is the weakest possible condition implying boundedness of characteristic classes:

\begin{theo}\label{mainintro}
	Let $G$ be a connected Lie group and let $R$ be its radical. 
	The following conditions are equivalent. 
	\begin{enumerate}
		\item All elements in the image of 
		$H^*(BG,\mathbb R)\to H^*(BG^{\delta},\mathbb R)$
		are bounded.
		\item The derived group $[R,R]$
		of the radical of $G$ is simply connected.
		
	\end{enumerate}
\end{theo}

In Section 2 we discuss stable commutator length and subgroup distortion
in Lie groups; the equivalent conditions stated in the
above theorem can also be expressed in those terms (Theorem \ref{theo: main}).
Section 3 deals with the primary obstruction to the existence of a global section 
of the universal $G$-bundle; this is the main tool for what follows. In Section 4 we prove the existence of
a non-zero lower bound on the stable commutator length 
of all elements in the commutator subgroup of the universal
cover of $G$ which are central and whose  stable commutator length is non-zero; this is a key point in the proof of
Theorem \ref{theo: main}. Section 5 is devoted to the
proof of Theorem \ref{theo: main}. Theorem \ref{mainintro} is obtained as a corollary of 
Theorem \ref{theo: main}.

We like to thank Jean Lafont for valuable discussions.

\section{Stable commutator length and distortion}

Let $G$ be a group and $[G,G]<G$ its
commutator subgroup (derived subgroup). The commutator length ${\rm cl}(z)$ of  $z\in [G,G]$
is the smallest number of commutators needed to express $z$.
The commutator length is sub-additive and the \emph{stable commutator length}
of $z\in [G,G]$ is defined as
\[
{\rm scl}_{G}(z)=\lim_{n\to\infty}{\rm cl}(z^n)/n.
\]
(We will drop the subscript and simply write $\rm scl(z)$, except when we want to emphasize the reference to $G$.) 
It is sub-additive on commuting elements. 
	
Let $H$ be a group generated by a finite symmetric set
$S$. The word length $|h|$ of $h\in H$ is the smallest number of elements of $S$ needed to express $h$.
If $(G,d)$ is a group with a left-invariant metric $d$, 
a finitely generated subgroup $H$ of  $G$ is \emph{distorted} in $(G,d)$ if
\[
	\inf_{h\in H\setminus\{e\}}\frac{d(e,h)}{|h|}=0.
\]
Being distorted does not depend on the choice of the generating set $S$.
An element $z\in G$ is \emph{distorted} if the subgroup it generates is distorted in $(G,d)$.
(Notice that torsion elements - including by convention
the identity element - are \emph{not} distorted.)
In the case $G$ is a connected Lie group and $d$ is the geodesic metric on $G$ induced by a left-invariant Riemannian metric, a  finitely generated subgroup $H$ of $G$, which is distorted with respect to $d$, is distorted with respect to
any left-invariant Riemannian metric.

The fundamental group $\pi_1(G)$ of a connected Lie group $G$ is finitely generated
and embeds naturally in the universal cover $\tilde G$ of $G$ as a central discrete subgroup.
It is proved in \cite{CMPS} that $\pi_1(G)<\tilde{G}$ is
undistorted if and only if every class $x^\delta$ in the image of 
$\Lambda: H^*(BG,\mathbb Z)\to H^*(BG^\delta,\mathbb Z)$
admits a bounded Borel cocycle representative $c: G\times\cdots\times
G\to \mathbb Z,\; [c]=x^\delta$ (meaning that the range of $c$ is bounded
and for every subset $S\subset \mathbb Z$, $c^{-1}(S)$ is a Borel
subset in the topological space $G\times\cdots\times G$). There are
examples of classes $x^\delta$ in the image of $\Lambda$,
which are bounded, but which do not admit a bounded Borel
cocycle representative. 
However, this can only happen, if the derived subgroup $[R,R]$ 
of the radical of $G$ is not closed in $G$ (see \cite{CMPS} for such
examples). 

The proof of the following theorem will be given in Section 5.

\begin{theo}\label{theo: main}
	Let $G$ be a connected Lie group and let $\tilde{G}$ be its universal cover.
	Let $R$ be the radical of $G$. The following conditions are equivalent. 
	\begin{enumerate}
		\item\label{cond: bounded} All elements in the image of 
		$H^*(BG,\mathbb Z)\to H^*(BG^{\delta},\mathbb Z)$
		are bounded.
		\item\label{cond: simply connected} $\pi_1([R,R])=\{0\}.$
		\item\label{cond: undistorted} All elements $z\in\pi_1(G)$ are undistorted 
		in $\tilde G$.
		\item\label{cond: stable} If $z\in\pi_1(G)\cap [\tilde G,\tilde G]$ 
		then either ${\rm scl}_{\tilde G}(z)>0$
		or $z$ has finite order. 
	\end{enumerate}
\end{theo}
 
In the statement of this theorem we used integer-valued group cohomology.   
It is well-known that all classes in the image of 
$H^*(BG,\mathbb R)\to H^*(BG^{\delta},\mathbb R)$
are bounded if and only if    
all classes in the image of 
$H^*(BG,\mathbb Z)\to H^*(BG^{\delta},\mathbb Z)$
are bounded; we recall the proof for the convenience of the reader. As the diagram of natural maps   
\[
	\xymatrix{
	H^*(BG,\mathbb Z)\ar[d]\ar[r]&H^*(BG,\mathbb R)\ar[d]\\
	H^*(BG^{\delta},\mathbb Z)\ar[r]&H^*(BG^{\delta},\mathbb R)
	}
	\]
commutes and as $H^*(BG,\mathbb Z)\otimes\mathbb R\cong H^*(BG,\mathbb R)$, the boundedness of integral classes implies the boundedness of the real ones. The converse also holds true because an integral class is bounded if and only if it is bounded when considered as a real class (apply \cite[Lemma 29]{CMPS} to $G^{\delta}$).
As a result, Theorem \ref{theo: main}
implies Theorem \ref{mainintro}.

\section{Primary obstruction in degree 2}

Let $G$ be a connected Lie group and let $\pi: P\to B$ be a principal $G$-bundle over a connected $CW$-complex $B$. 
We denote by
$$o(P)\in H^2(B,\pi_1(G))$$
the primary obstruction to the existence of a section for $\pi$.

Let $z\in\pi_1(G)$ be an element of the commutator
subgroup of the universal cover of $G$. Let $g={\rm cl}(z)$ be its
commutator length and let $x_1,\dots,x_g$, 
	$y_1, \dots,y_g$ be in the universal cover such that
	\[
		\prod_{n=1}^g[x_n,y_n]=z.
	\]
	The $2g$ elements of $G$ obtained
	by projecting $x_1,\dots,x_g,y_1,\dots,y_g$ to $G$ define a 
	homomorphisms $\pi_1(\Sigma_g)\to G$ 
	of the fundamental group of the closed oriented surface $\Sigma_g$ of genus $g$ 
	to $G$. Let $P\to\Sigma_g$ be the associated flat $G$-bundle over $\Sigma_g$. By construction
	\begin{equation}\label
{equa: formula}
		o(P)([\Sigma_{{\rm cl}(z)}])=z,	
	\end{equation}
	(see Milnor \cite[Lemma 2]{Mil} and Goldman \cite{Gol}).
	
Let $EG\to BG$ be the universal $G$-bundle. Since $G$ is
connected, $BG$ is simply connected and therefore
for any abelian group $A$
there are natural isomorphisms
\begin{align*}
H^2(BG,A)&\cong {\rm Hom}(H_2(BG,\mathbb Z),A)\\
                &\cong {\rm Hom}(\pi_2(BG),A)\\
                &\cong {\rm Hom}(\pi_1(G),A). 	
\end{align*}
In the case of $A=\pi_1(G)$, the universal class $o(EG)$ corresponds 
to ${\rm id}_{\pi_1(G)}\in {\rm Hom}(\pi_1(G),A)$.

Let $o^{\delta}(EG)$ denote the image of $o(EG)$ under the canonical
map $H^2(BG,\pi_1(G))\to H^2(BG^{\delta},\pi_1(G))$.

\begin{prop}\label{prop: bounded} Let $G$ be a connected Lie group. The following conditions are equivalent.
	\begin{enumerate}
		\item The class $o^{\delta}(EG)$ is bounded.
		\item All classes in the image of $H^2(BG,\mathbb Z)\to H^2(BG^{\delta},\mathbb Z)$\\
		are bounded. 
	\end{enumerate}
	
\end{prop}

\begin{proof}
Let $x\in H^2(BG,\mathbb Z)\cong {\rm Hom}(\pi_1(G),\mathbb Z)$ with
corresponding homomorphisms $\phi_x: \pi_1(G)\to \mathbb Z$.
By construction, the induced coefficient homomorphism $(\phi_x)_*: H^2(BG,\pi_1(G))\to
H^2(BG,\mathbb Z)$  maps $o(EG)$ to $x$. If $x^\delta$ denotes the image of $x$ in
$H^2(BG^\delta,\mathbb Z)$, we conclude by naturality that
$(\phi_x)_*(o^\delta(EG))=x^\delta$.
Because $\phi_x$ maps bounded sets to bounded sets, we conclude that (1) implies (2).
Conversely, because $\pi_1(G)$ is a finitely generated abelian group, (2) implies
that all classes in the image of $H^2(BG,\pi_1(G))\to H^2(BG^\delta,\pi_1(G))$ are
bounded and in particular, the universal class $o^\delta(EG)$ is bounded.
\end{proof}


If $x\in H^*(BG^{\delta},A)$, we denote $x_{\mathbb R}$ its image in $H^*(BG^{\delta},A\otimes\mathbb R)$. The class $o^\delta_{\mathbb R}(EG)$ lies in the subgroup $H^2(BG^\delta, \pi_1([G,G])\otimes\mathbb R)$.

\begin{prop}\label{prop: finite cover} Let $p:G\to Q$ be a finite homomorphic cover of connected Lie groups.
Then $p^*: H^*(BQ^\delta,\mathbb R)\to H^*(BG^\delta,\mathbb R)$ is an isomorphism and
maps the subgroup of bounded classes of $H^*(BQ^\delta,\mathbb R)$
bijectively onto the subgroup of bounded classes in $H^*(BG^\delta,\mathbb R)$.
As a result, $o^\delta(EG)$ is bounded, if and only if $o^\delta(EQ)$ is bounded.
\end{prop}

\begin{proof}

Let $F$ be the kernel of $p$. Since $F$ is a finite group, $H^n(BF,\mathbb R)=0$ for $n>0$ and
therefore, $p^*: H^*(BQ^\delta,\mathbb R)\to H^*(BG^\delta,\mathbb R)$ is an isomorphism. Because
$F$ is amenable, the induced map of bounded cohomology
groups $p^*_b: H^*_b(Q^\delta,\mathbb R)\to H^*_b(G^\delta,\mathbb R)$ is an isomorphism too.
Thus $p^*$ induces an isomorphism between the subgroup of bounded elements in $H^*(BQ^\delta,\mathbb R)$
and those in $H^*(BG^\delta,\mathbb R)$, proving the first part of the assertion. 
Let $\iota: \pi_1(G)\otimes\mathbb R \to \pi_1(Q)\otimes\mathbb R$ be the
quasi-isometric isomorphism induced by the inclusion of $\pi_1(G)$ in $\pi_1(Q)$.
Since
$$p^*o^\delta_{\mathbb R}(EQ)=\iota_*o^\delta_{\mathbb R}(EG),$$ 
one of the two is bounded if and only if the other one is. It follows that
$o^\delta(EQ)$ is bounded if and only if $o^\delta(EG)$ is.
  \end{proof}

The map $H^*(BG,\mathbb Z)\to H^*(BG^{\delta},\mathbb Z)$ is injective (see Milnor, \cite{MilDis}).
Because $\pi_1(G)$ is  a finitely generated  abelian group, it follows that
$H^2(BG,\pi_1(G))\to H^2(BG^\delta,\pi_1(G))$ is injective too.
Therefore, $o^{\delta}(EG)$ is zero if 
and only if $G$ is simply connected.

\begin{prop}\label{prop: scl>0} Assume $o^{\delta}(EG)$ is bounded
and let $z\in\pi_1(G)$ be an element of infinite order in the
	commutator subgroup of the universal cover of $G$. 
	Then
	\[
		{\rm scl}(z)\geq \frac{1}{4\|o^{\delta}(EG)\|_{\infty}}\lim_{n\to\infty} \frac{|z^n|}{n}>0.
	\]
\end{prop}

\begin{proof} 
An element of infinite order in a finitely generated abelian group is undistorted hence $\lim_{n\to\infty} |z^n|/n>0$.
For any $n\neq 0$, $g:={\rm cl}(z^n)\geq 1$.	Representing $\Sigma_g$ as the quotient of a 
$4g$-gon in the usual way and decomposing the $4g$-gon into $4g-2$ triangles (by coning over a vertex), 
Formula (\ref{equa: formula}) yields
	\[
		|z^n|\leq \|o^{\delta}(EG)\|_{\infty}(4{\rm cl}(z^n)-2).
	\]	
The result then follows by dividing by $n$ and taking limits.	 
\end{proof}

Because $o^\delta(EG)$ is bounded for connected semisimple
Lie groups $G$ (\cite[Proposition 47, Theorem 1]{CMPS}), we have the following corollary.

\begin{coro}{\label{ss-case}}
Let $G$ be a simply connected semisimple Lie group and $z\in G$
a central element of infinite order. Then ${\rm scl}(z)>0$.
\end{coro}

\begin{rema} Borel proved in \cite{Borel} that {\rm cl} is bounded on 
connected semisimple Lie groups with finite center. Thus {\rm scl} is
the zero function on such groups. 
\end{rema}

As an immediate consequence of Formula (\ref{equa: formula}) we also obtain:

\begin{lemma}\label{lemma: Hopf} Let $\alpha\in H_2(BG^{\delta},\mathbb Z)$.
Let $g(\alpha)$ denote the minimal genus of a surface representing $\alpha$ (cf.~Hopf, \cite[Satz IIa]{Hopf} or Thom, \cite{Thom} or Calegari, \cite[1.1.2 Example 1.4]{Cal}).
Then $o^\delta(EG)(\alpha)\in \pi_1(G)$, viewed as an element of the universal cover of $G$, lies
in the commutator subgroup of $\tilde{G}$ and 
\[
{\rm cl}(o^{\delta}(EG)(\alpha))\leq g(\alpha).
\]
\hfill$\square$
\end{lemma}

For an element $x\in H_n(BG^\delta,\mathbb R)$ we write $\|x\|_1$ for its $\ell_1$-seminorm
(cf. Gromov \cite{GroVol}).

\begin{prop}\label{prop: l^1} Let $\alpha\in H_2(BG^{\delta},\mathbb Z)$ and let $\alpha_{\mathbb R}$ 
denote its image in $H_2(BG^{\delta},\mathbb R)$. Then
\[ 
{\rm scl}(o^{\delta}(EG)(\alpha))\leq \frac{\|\alpha_{\mathbb R}\|_1}{4}.
\]
\end{prop}

\begin{proof} Lemma \ref{lemma: Hopf} implies

	\begin{equation*}
	{\rm scl}(o^{\delta}(EG)(\alpha))=\lim_{n\to\infty}\frac{{\rm cl}(o^{\delta}(EG)(n\alpha))}{n}
	\leq\lim_{n\to\infty}\frac{g(n\alpha)}{n}
	\end{equation*}
	and
	from Barge-Ghys \cite[Lemme 1.5]{BG} we infer that
	$$\lim_{n\to\infty}\frac{g(n\alpha)}{n}=\frac{\|\alpha_{\mathbb R}\|_1}{4},$$
	finishing the proof.
	
\end{proof}

\section{A lower bound for the stable commutator length}

\begin{prop}\label{prop: s>0} Let $G$ be a simply connected Lie group.
	There exists a constant $s>0$ such that if $z\in [G,G]$ is
	central in $G$ and ${\rm scl}(z)<s$ then ${\rm scl}(z)=0$.
\end{prop}

\begin{proof} Let $G=RL$ be a Levi decomposition.
	Let $z=r\ell$ with $r\in R$ and $\ell\in L$. Because this decomposition
	of
	$z$ is unique and $R$ is normal in $G$, $\ell$ is central in $L$ and $L$ centralizes $r$.
	First we show that $r\in [R,R]$.
	
	\noindent
	Case 1: $R$ is commutative. In this case, we show that $r=e$.
	As $L$ is semisimple, $r$ lies in a direct factor of $G$ and, because $\ell\in L=[L,L]$,
	$r$ lies in $[G,G]$, thus $r$ must be trivial.
	
	\noindent
	Case 2: $R$ is not commutative. Let $p:RL \to (R/[R,R])L$ be the projection. Then $p(z)=p(r)\ell$ and by Case 1,
	$p(r)=e$, which implies $r\in [R,R]$.
	
	Because the stable commutator length vanishes for connected solvable Lie groups
	(cf. Bavard \cite[]{Bav}) we have 
	\[
		{\rm scl}(r)\leq {\rm scl}_R(r)=0.
	\]
	By sub-additivity on commuting elements it follows that ${\rm scl}(z)={\rm scl}(\ell)$.
	But ${\rm scl}(\ell)={\rm scl}_L(\ell)$. The quotient $Q$ of $L$ by its center
	is a (linear) semisimple group. Therefore
	$o^\delta(EQ)$ is bounded \cite[Proposition 47, Theorem 1]{CMPS}.
	Using Proposition \ref{prop: scl>0} we obtain then an $s>0$
	such that ${\rm scl}(z)\ge s$ if ${\rm scl}(z)\ne 0$,
	completing the proof of the proposition. 
	
\end{proof}

\begin{prop} Let $G$ be a connected Lie group. The function ${\rm scl}$ is continuous on any closed abelian subgroup of $[G,G]$ and
satisfies ${\rm scl}(\exp(tX))=t\cdot{\rm scl}(\exp(X))$ for $X$ in the Lie algebra
of $[G,G]$ and $t\ge 0$.
\end{prop}

\begin{proof}
There is a compact neighborhood $K$ of $e\in [G,G]$ such that for all $k\in K$,
	${\rm cl}(k)\le N$, where $N$ is the dimension of the Lie algebra of $[G,G]$
	(Bourbaki, \cite[page 398, 10]{Bou}). Therefore,
	{\rm cl} is bounded on compact subsets of $[G,G]$. 
	Assume that $x\in [G,G]<G$ lies on a one-parameter subgroup $\exp(X)=x$. Consider the  fractional and integer parts $1<t=\{t\}+\lfloor t\rfloor$. As $\lim_{t\to\infty}\lfloor t\rfloor/t=1$, and as
		\[
			\left|\frac{{\rm cl}(\exp(\lfloor t\rfloor X))}{\lfloor t\rfloor}-\frac{{\rm cl}(\exp(tX))}{t}\right|\leq
			2\cdot\frac{{\rm cl}(\exp(\{t\}X))}{t},
		\]
	we deduce that 
	\[
		\lim_{t\to\infty}\frac{{\rm cl}(\exp(tX))}{t}=\lim_{n\to\infty}\frac{{\rm cl}(x^n)}{n}={\rm scl}(x).
	\]
Hence for all $\exp(X)\in [G,G]$, $t\geq 0$, we have $${\rm scl}(\exp(tX))=t\cdot{\rm scl}(\exp(X))\,.$$
As 
${\rm scl}\le {\rm cl}$, ${\rm scl}$ is also bounded on compact sets. This implies 
	that  ${\rm scl}$ is continuous at the identity.  Sub-additivity on commuting elements and continuity at the identity imply continuity on any closed abelian subgroup of $[G,G]$.
\end{proof}	

\begin{lemma}\label{lemma: epsilon} Let $G$ be a simply connected Lie group. Let $V\cong\mathbb R^r$ be a closed
	subgroup of $[G,G]$ and let $A\cong\mathbb Z^r$ be discrete in $V$ and central in $G$.
	Assume ${\rm scl}(a)>0$ for all $0\neq a\in A$. Then there is a constant $\varepsilon>0$ such that 
	for all $a\in A$, ${\rm scl}(a)\geq \varepsilon|a|$. 
\end{lemma}
	
\begin{proof} As ${\rm scl}$ is continuous on $V$ and linear along one-parameter semigroups,
	it is enough to show that ${\rm scl}$ vanishes nowhere on the unit sphere $S$ of $V$.
	Assume there is $z\in S$ with ${\rm scl}(z)=0$. As $A$ is cocompact in $V$,
	for each $n\in\mathbb N$, there is $a_n\in A\setminus\{0\}$ and $t_n>0$ such that
	\[
		|a_n-t_nz|<1/n.
	\]
	Hence
	\[
		0<{\rm scl}(a_n)={\rm scl}(a_n-t_nz)\to 0,
	\]
	contradicting Proposition \ref{prop: s>0}.
\end{proof}	
		 
\section{Proof of the main theorem}
We use the notation introduced at the beginning of Section 3.

\begin{lemma}\label{lemma: Hahn-Banach} Let $G$ be a connected Lie group. Let $\tilde{G}$ be its universal cover. Assume that $A=\pi_1(G)\cap[\tilde{G},\tilde{G}]\cong \mathbb Z^r$ is discrete in a closed subgroup $V\cong\mathbb R^r$ of $[\tilde{G},\tilde{G}]$. If ${\rm scl}(a)>0$ for all $a\in A\setminus\{0\}$, then $o^{\delta}(EG)$ is bounded.	
\end{lemma}

\begin{proof} Let us denote $X=BG^{\delta}$. Let $C_2(X,\mathbb R)$ be the
space of singular 2-chains, endowed with the $\ell_1$-norm.
The free $\mathbb Z$-module $Z_2(X,\mathbb Z)$ of integral cycles contains a basis of the $\mathbb R$-vector space 
$Z_2(X,\mathbb R)$ of cycles. Hence there is a unique
$\mathbb R$-linear map $c:Z_2(X,\mathbb R)\to V$ such that for all 
	$x\in Z_2(X,\mathbb Z)$,
	$c(x)=o^{\delta}(EG)([x])$. Let $z\in Z_2(X,\mathbb Q)$, $z=\sum_ir_i\sigma_i$, 
	$r_i\in\mathbb Q$, $\sigma_i$ a singular $2$-simplex of $X$, 
	$|z|_1=\sum_i|r_i|$. 
	Let $m\in\mathbb N$, such $mr_i\in\mathbb Z$ for all $i$. Then, with $\varepsilon>0$ as 
	in Lemma \ref{lemma: epsilon},

	\[
	|c(z)|=\frac{1}{m}|c(mz)|=\frac{1}{m}|o^{\delta}(EG)([mz])|\leq\frac{{\rm scl}(o^{\delta}(EG)([mz]))}{m\varepsilon}
	\]
	and because of Proposition \ref{prop: l^1},
	\[
	\frac{{\rm scl}(o^{\delta}(EG)([mz]))}{m\varepsilon}	
	\leq
	\frac{\|[mz]_{\mathbb R}\|_1}{4m\varepsilon}\leq\frac{|z|_1}{4\varepsilon}.   	
	\]
	As $Z_2(X,\mathbb Q)\subset Z_2(X,\mathbb R)$ is dense, the norm of the linear map
	$c$ is bounded by $(4\varepsilon)^{-1}$. Hahn-Banach's theorem provides a linear extension
	$\hat{c}$
	of $c$ to all of $C_2(X,\mathbb R)$ with the same bound on the norm. Hence 
	\[
	[\hat{c}]=o^{\delta}_{\mathbb R}(EG)\in H^2(X,A\otimes\mathbb 
	R)\subset H^2(X,\pi_1(G)\otimes\mathbb R) 
	\] 
	is bounded. Hence $o^{\delta}(EG)\in H^2(X,\pi_1(G))$ is bounded as well.
\end{proof}

\noindent
We are ready for the proof of the main theorem.
\begin{proof}[Proof of Theorem \ref{theo: main}]
\smallskip
(\ref{cond: bounded}) implies (\ref{cond: simply connected}):

\noindent
Assume that $\pi_1([R,R])\ne 0$. Then according to Goldman \cite{Gol},
and as we have seen at the beginning of Section 3, there exists a map $f:\Sigma_g\to BR^\delta$
for some closed oriented
surface of genus $g\ge 1$ such that $(f^*o^\delta(ER))([\Sigma_g])\in \pi_1([R,R])$ is not zero.
Because $\pi_1([R,R])$ is torsion-free, we conclude that the image
$o^\delta_{\mathbb R}(ER)\in H^2(BR^\delta,\pi_1(R)\otimes\mathbb R)$ also satisfies 
$(f^*o^\delta_{\mathbb R}(ER))([\Sigma_g])\ne 0$.
It follows that $o^\delta_{\mathbb R}(ER)$ cannot be bounded, because the bounded
cohomology of the discrete amenable group $R^\delta$ vanishes (Johnson, \cite{John}).
Denote by $\iota: \pi_1(R)\otimes \mathbb R\to \pi_1(G)\otimes \mathbb R$ the
natural isometric embedding and let $j: R\to G$ be the inclusion. 
Because
$$0\neq\iota_*(o^\delta_\mathbb R(BR))=j^*(o^\delta_\mathbb R(BG))$$
we see that $o^\delta_\mathbb R(EG)$ is not bounded.
It follows that $o^\delta(EG)$ cannot be bounded either.
\smallskip

(\ref{cond: simply connected}) implies (\ref{cond: undistorted}):

\noindent
Assume that $z\in\pi_1(G)\subset\tilde{G}$ is distorted in $\tilde{G}$ and
hence of infinite order. Then there is 
an $n>0$ such that 
$z^n\in \tilde{R}$, otherwise the projection of $z$ in $\tilde{G}/\tilde{R}$ would be
a central distorted element and this is impossible \cite[Lemma 6.3]{CPS}.
The distorted element  $z^n\in\tilde{R}\cap \pi_1(G)=\pi_1(R)$
must belong to $[\tilde{R},\tilde{R}]$ because otherwise its projection to 
\[
	\tilde{R}/[\tilde{R},\tilde{R}]\subset\tilde{G}/[\tilde{R},\tilde{R}]
\]
would span a central distorted line. But a line in $\tilde{R}/[\tilde{R},\tilde{R}]\cong\mathbb R^r$ which is central in 
$\tilde{G}/[\tilde{R},\tilde{R}]$ is a direct factor because $\tilde{G}/\tilde{R}$
is semisimple. Hence it cannot be distorted. We conclude that 
\[z^n\in [\tilde{R},\tilde{R}]\cap\pi_1(G)=\pi_1([R,R]),
\]
contradicting (\ref{cond: simply connected}).

\smallskip

(\ref{cond: undistorted}) implies (\ref{cond: stable}):

\noindent
Let
$\tilde{G}=\tilde{R}L$ be a Levi decomposition and let 
$z=r\ell\in\pi_1(G)\cap [\tilde{G},\tilde{G}]$ be of infinite order,
$r\in{\tilde R}$ and $\ell\in L$.
Note that $L=[L,L]$ and ${\rm scl}(z)\ge {\rm scl}(\ell)={\rm scl}_L(\ell)$.
If $\ell$ has infinite order, Corollary \ref{ss-case} implies that
${\rm scl}_L(\ell)>0$ and we conclude that ${\rm scl}(z)>0$. If $\ell$ has
finite order, there is an $n>0$ such that $z^n=(r\ell)^n=r^n$,
and $r^n\in \tilde{R}\cap\pi_1(G)$ actually lies in $[\tilde{R},\tilde{R}]$, as we have
seen in the course of the proof of \ref{prop: s>0}. 
Proposition 19 of \cite{CMPS} implies that $z$ is distorted
in $\tilde{R}$, hence also in $\tilde{G}$.

\smallskip
(\ref{cond: stable}) implies (\ref{cond: bounded}):
\smallskip

\noindent
Let $K$ be a maximal compact subgroup of $[G,G]$. The universal cover of $K$ is a closed subgroup $E\times S$ of $[\tilde{G},\tilde{G}]$ where $E\cong\mathbb R^r$ and $S$ is  compact semisimple.
Let $p:E\times S\to S$ be the projection onto the second factor.
The projection of $\pi_1(G)\cap[\tilde{G},\tilde{G}]=\pi_1([G,G])\subset E\times S$
in $S$ is central hence finite. Thanks to Proposition \ref{prop: finite cover}, we may replace $G$ by a finite homomorphic cover and hence 
assume that $\pi_1(G)\cap[\tilde{G},\tilde{G}]\subset E$. Applying Lemma \ref{lemma: Hahn-Banach} and Proposition \ref{prop: bounded} implies boundedness in degree $2$.
Boundedness in higher degree then follows from \cite[Lemma 51, Theorem 54]{CMPS}.

\end{proof}


\begin{bibdiv}
\begin{biblist}
	
	\bib{BG}{article}{
	   author={Barge, Jean},
	   author={Ghys, {\'E}tienne},
	   title={Surfaces et cohomologie born\'ee},
	   language={French},
	   journal={Invent. Math.},
	   volume={92},
	   date={1988},
	   number={3},
	   pages={509--526},
	}

	\bib{Bav}{article}{
	   author={Bavard, Christophe},
	   title={Longueur stable des commutateurs},
	   language={French},
	   journal={Enseign. Math. (2)},
	   volume={37},
	   date={1991},
	   number={1-2},
	   pages={109--150},
	}

\bib{Borel}{article}{
 author={Borel, Armand},
	   title={Class functions, conjugacy classes and commutators in
	   semisimple Lie groups},
	   journal={Australian Math. Soc Lecture Series},
	   volume={9},
	   date={1997},
	   pages={1--19},
     publisher={Cambridge University Press}
     }
	
	\bib{Bou}{book}{
	   author={Bourbaki, Nicolas},
	   title={Lie groups and Lie algebras. Chapters 1--3},
	   series={Elements of Mathematics (Berlin)},
	   note={Translated from the French;
	   Reprint of the 1989 English translation},
	   publisher={Springer-Verlag},
	   place={Berlin},
	   date={1998},
	   pages={xviii+450},
	}	
	
	
\bib{Bucher}{article}{
author={Bucher-Karlsson, Michelle},
title={Finiteness properties of characteristic classes of flat bundles},
journal={Enseign. Math. },
volume={53},
date={2007},
number={2},
pages={33-66}
}

\bib{Cal}{book}{
   author={Calegari, Danny},
   title={scl},
   series={MSJ Memoirs},
   volume={20},
   publisher={Mathematical Society of Japan},
   place={Tokyo},
   date={2009},
   pages={xii+209},
}

\bib{CMPS}{article}{
		   author={Chatterji, Indira},
		   author={Mislin, Guido},
		   author={Pittet, Christophe},
		   author={Saloff-Coste, Laurent},
		   title={A geometric criterion for the boundedness of characteristic classes},
		   journal={Math. Ann.},
		   volume={351},
		   date={2011},
		   number={3},
		   pages={541--569},
		}	
		\bib{CPS}{article}{
		   author={Chatterji, Indira},%
		   author={Pittet, Christophe}
		   author={Saloff-Coste, Laurent},
		   title={Connected Lie groups and property RD},
		   journal={Duke Math. J.},
		   volume={137},
		   date={2007},
		   number={3},
		   pages={511--536},
		}
		
	\bib{Dup}{article}{
		   author={Dupont, Johan L.},
		   title={Bounds for characteristic numbers of flat bundles},
		   conference={
		      title={Algebraic topology, Aarhus 1978 (Proc. Sympos., Univ. Aarhus,
		      Aarhus, 1978)},
		   },
		   book={
		      series={Lecture Notes in Math.},
		      volume={763},
		      publisher={Springer},
		      place={Berlin},
		   },
		   date={1979},
		   pages={109--119},
		}	

	\bib{Gol}{article}{
	   author={Goldman, William M.},
	   title={Flat bundles with solvable holonomy. II. Obstruction theory},
	   journal={Proc. Amer. Math. Soc.},
	   volume={83},
	   date={1981},
	   number={1},
	   pages={175--178},
	}
	
	\bib{GroVol}{article}{
	   author={Gromov, Michael},
	   title={Volume and bounded cohomology},
	   journal={Inst. Hautes \'Etudes Sci. Publ. Math.},
	   number={56},
	   date={1982},
	   pages={5--99 (1983)},
	}
	
	\bib{Hopf}{article}{
	   author={Hopf, Heinz},
	   title={Fundamentalgruppe und zweite Bettische Gruppe},
	   language={German},
	   journal={Comment. Math. Helv.},
	   volume={14},
	   date={1942},
	   pages={257--309},
	}

	\bib{John}{book}{
	   author={Johnson, Barry Edward},
	   title={Cohomology in Banach algebras},
	   note={Memoirs of the American Mathematical Society, No. 127},
	   publisher={American Mathematical Society},
	   place={Providence, R.I.},
	   date={1972},
	   pages={iii+96},
	}



	
	\bib{Mil}{article}{
	  author={Milnor, John},
	   title={On the existence of a connection with curvature zero},
	   journal={Comment. Math. Helv.},
	   volume={32},
	   date={1958},
	   pages={215--223},
	}
	
	\bib{MilDis}{article}{
	   author={Milnor, John},
	   title={On the homology of Lie groups made discrete},
	   journal={Comment. Math. Helv.},
	   volume={58},
	   date={1983},
	   number={1},
	   pages={72--85},
	}
	\bib{Thom}{article}{
	   author={Thom, Ren{\'e}},
	   title={Sur un probl\`eme de Steenrod},
	   language={French},
	   journal={C. R. Acad. Sci. Paris},
	   volume={236},
	   date={1953},
	   pages={1128--1130},
	}
	 
	
	
\end{biblist}
\end{bibdiv}
\end{document}